\documentclass{article}
\usepackage{amsmath}
\usepackage{amsfonts,amsmath,amsthm,amssymb}
\usepackage{mathtools}
\usepackage[colorlinks,citecolor=blue]{hyperref}
\usepackage[sort, numbers]{natbib}

\usepackage{geometry}
\usepackage{pgf,pgfplots}
\usepackage{caption}
\usepackage{subcaption}
\usepackage[]{enumitem}
\usepackage{bigints}
\usepackage{float}

\usepackage[title]{appendix}

\newtheorem{lem}{Lemma}[section]
\newtheorem{cor}[lem]{Corollary}
\newtheorem{prop}[lem]{Proposition}
\newtheorem{remark}[lem]{Remark}
\newtheorem{thm}[lem]{Theorem}
\newtheorem{defn}[lem]{Definition}
\newtheorem{eg}[lem]{Example}

\newtheorem{assum}[lem]{Assumption}

\newtheorem*{prop*}{Proposition}
\newtheorem*{thm*}{Theorem}
\newtheorem*{def*}{Definition}
\newtheorem*{lem*}{Lemma}

\newcommand{\R}{\mathbb{R}}
\renewcommand{\P}{\mathbb{P}}
\newcommand{\I}{\mathbb{I}}

\newcommand{\F}{\mathcal{F}}
\newcommand{\N}{\mathbb{N}}

\makeatletter
\renewenvironment{proof}[1][\proofname] {\par\pushQED{\qed}\normalfont\topsep6\p@\@plus6\p@\relax\trivlist\item[\hskip\labelsep\bfseries#1\@addpunct{:}]\ignorespaces}{\popQED\endtrivlist\@endpefalse}
\makeatother

\renewcommand{\( }{\left(}
\renewcommand{\)}{\right)}
\pgfplotsset{compat=1.16}

\DeclareMathOperator*{\diver}{div}
\DeclareMathOperator*{\capac}{Cap}
\DeclareMathOperator*{\tr}{Tr}
\DeclareMathOperator*{\supp}{supp}
\allowdisplaybreaks

\AtBeginDocument{
	\label{CorrectFirstPageLabel}
	
}

\begin{document}
\title{On A Class Of Rank-Based Continuous Semimartingales \footnote{The first author acknowledges the support of the Natural Sciences and Engineering Research Council of Canada (NSERC).}}
\author{David Itkin\thanks{Department of Mathematical Sciences, Carnegie Mellon University, Wean Hall, 5000 Forbes Ave, Pittsburgh, Pennsylvania 15213, USA, \url{ditkin@andrew.cmu.edu}.} \and Martin Larsson\thanks{Department of Mathematical Sciences, Carnegie Mellon University, Wean Hall, 5000 Forbes Ave, Pittsburgh, Pennsylvania 15213, USA, \url{martinl@andrew.cmu.edu}.}}
\maketitle
\begin{abstract}
Using the theory of Dirichlet forms we construct a large class of continuous semimartingales on an open domain $E \subset \R^d$, which are governed by rank-based, in addition to name-based, characteristics. Using the results of Baur et al.\ \cite{baur2013construction} we obtain a strong Feller property for this class of diffusions. As a consequence we are able to establish the nonexistence of triple collisions and obtain a simplified formula for the dynamics of its rank process. We also establish conditions under which the process is ergodic. Our main motivation is Stochastic Portfolio Theory (SPT), where rank-based diffusions of this type are used to model financial markets. We show that three main classes of models studied in SPT -- Atlas models, generalized volatility-stabilized models and polynomial models -- are special cases of our framework.
\end{abstract}

\bigskip



\section{Introduction and Main Results}
We fix $d \geq 2$ and work on a nonempty connected open domain $E \subseteq \R^d$. Taking as inputs an instantaneous covariation matrix $c$ and a density function $p$ (satisfying Assumption~\ref{ass:inputs} below) we construct a particle system represented by a continuous semimartingale $X = (X_1,\dots,X_d)$ where the dynamics of each coordinate depend on its rank relative to the other coordinates.

Our approach is to use the theory of Dirichlet forms, with absorption at the boundary, using which we obtain existence of the continuous semimartingale $X$. Using the results of \cite{baur2013construction} we obtain that $X$ possesses the \emph{$L^p$-strong Feller property}. This allows us to analyze path properties of the particle system and prove that \emph{triple collisions} do not occur in this system; that is with probability one there is no positive time $t$ such that $X_i(t) = X_j(t) = X_k(t)$ for any distinct indices $i,j,k \in \{1,\dots,d\}$. Under a nonexplosion assumption on the process we obtain conditions when the process is ergodic and, as a consequence, obtain asymptotic rates on how much time each particle spends at each rank.  

The main motivation of this work is due to the important role that rank-based diffusions play in Stochastic Portfolio Theory (SPT). SPT was introduced by Fernholz in \cite{fernholz1999diversity,fernholz2002stochastic} as a descriptive theory with the goal of explaining observable market phenomena. An important observation in SPT is that the ranked relative market capitalizations, called the \emph{market weights}, have remained remarkably stable over time across different US equity markets \cite[Chapter~5]{fernholz2002stochastic}. This has spurred an interest in ergodic diffusions on the simplex whose dynamics evolve according to the ranks of the market weights. Three important classes of models include (generalized) volatility-stabilized models introduced in \cite{fernholz2005relative,pickova2014generalized}, (hybrid) Atlas models of \cite{banner2005atlas,ichiba2011hybrid} and polynomial diffusion models studied in \cite{cuchiero2019polynomial,MR3551857}.

 Recently, in \cite{kardaras2018ergodic} and later in \cite{itkin2020robust}, the authors worked with a certain class of ergodic diffusions, modelling the market weights, which were used to study an asymptotic robust growth optimization problem in the context of SPT. The approach taken there provides a systematic way to construct ergodic models for the market weight processes, which do not rely on very concrete specifications like the aforementioned examples, but rather allow for more general constructions only using the inputs $(c,p)$. The techniques employed in those studies, however, were unable to handle certain rank-based models, such as the Atlas model, due to the discontinuous characteristics that appear in the drift of the market weight process. Our construction extends the previous class to accommodate a richer family of rank-based models. In particular, we show that the Atlas model of \cite{banner2005atlas}, the volatility-stabilized market of \cite{fernholz2005relative} and the polynomial models of \cite{cuchiero2019polynomial} are special cases of the theory developed in this paper. 

Using the Dirichlet form structure we are also able to show that triple collisions do not occur in these models. The existence of triple collisions is an undesirable occurrence in many cases. Indeed, in \cite{MR3055258} the authors show that the SDE corresponding to a \emph{competing Brownian particle system} (that is, a system where the drift and volatility coefficients are rank-dependent and piecewise constant) has a strong solution up to the first triple collision time. In \cite{banner2008local} the authors show that the \textit{rank process} (also known as the \textit{order statistics}) of a semimartingale is again a semimartingale and obtain a general formula for its dynamics. This formula for the semimartingale's rank processes, however, simplifies considerably when the set of triple collision points is polar. This simplified formula leads to a more tractable analysis of rank-based portfolios in the context of SPT. As such, it is desirable to develop methods for determining whether or not triple collisions occur. The authors of \cite{MR3055258,MR2680554} derived a sufficient condition for triple collisions to not occur with probability one in the setting of competing Brownian particle systems with a bounded measurable drift, while \cite{MR3325099} showed that the condition of \cite{MR3055258} was both necessary and sufficient. Results concerning simultaneous collisions and collisions of four of more particles were obtained in \cite{MR3706753,MR3607803}. 

To the best of our knowledge previous approaches to study the behaviour of such systems have been limited to competing Brownian particle systems. In these cases the study of triple collisions was carried out using a very careful hands-on analysis of reflected Brownian motion and Bessel processes as much is known about the path properties of those processes. Our approach uses ideas from potential theory, most notably the notion of zero capacity sets, to obtain the nonexistence of triple collisions. This more abstract perspective allows us to handle cases beyond the competing Brownian particle systems described above; namely we analyze diffusions living on open subsets of $\R^d$ and evolving with drift and volatility characteristics that need not be piecewise constant. On the other hand, our approach is limited to continuous volatility structures and is unable to handle some of the volatility specifications for Brownian particle systems that were shown to be devoid of triple collisions in \cite{MR3055258}. A more nuanced discussion of the advantages and limitations of our approach is carried out in Section~\ref{sec:conclusion}.

Aside from the study of triple collisions, rank-based diffusions have garnered a lot of attention in the literature. For instance, Pal and Pitman \cite{MR2473654} study both a finite and infinite dimensional system of one-dimensional Brownian particle systems with rank-dependent drifts and obtain convergence to a stationary distribution under the total variation norm. Pal and Shkolnikov in \cite{MR3211002} obtain exponential concentration of measure results for such Brownian particle systems. Other recent works studying Brownian particle systems include \cite{MR2884054,MR2914770,MR3055264,MR3399811,MR3513592,MR3729655,MR3055258}.

The paper is organized as follows. Section~\ref{sec:Dirichlet_form} contains preliminaries including the Dirichlet form and integration by parts formula. In Section~\ref{sec:construction} we obtain existence of the process $X$ and establish the $L^p$-strong Feller property. Sections~\ref{sec:collisions} and \ref{sec:ergodicty} contain our main results including the nonexistence of triple collisions, conditions for ergodicity and formulas for long-term occupation times of the process. We also consider a few examples. Section~\ref{sec:SPT} contains applications to SPT including examples which relate the class of models constructed in this paper to those previously studied in the literature. Lastly, in Section~\ref{sec:conclusion} we conduct a detailed discussion of the limitations of our approach. 

\subsection{Notation}
For $d \geq 2$ we denote by $\boldsymbol{1}$ the all ones vector in $\R^d$. For a measurable set $E \subset \R^d$ we write $\mathcal{B}(E)$ and $\mathcal{B}_b(E)$ for the Borel subsets of $E$ and the collection of Borel measurable bounded functions on $E$ respectively. We denote by $\mathbb{S}^d_{++}$ the set of all symmetric positive definite $d \times d$ matrices. For $A \in \mathbb{S}^d_{++}$ we define its trace and kernel by $\tr(A) = \sum_{i=1}^d A_{ii}$ and Ker$(A)= \{x \in \R^d: Ax = 0\}$ respectively. For $k \geq 1$, an open set $U \subseteq \R^d$ and a vector space $V$, we will denote by $C^k(U;V)$ the set of all $k$-times continuously differentiable functions from $U$ to $V$. $C^0(U;V)$ is used to denote the set of continuous functions. The inclusion of a subscript `$c$' denotes the set of functions in that class with compact support. When $V = \R$ we omit it from the notation and simply write $C^k(U)$ and $C^k_c(U)$ instead of $C^k(U;\R)$ and $C_c^k(U;\R)$ respectively.  $\mathcal{T}$ denotes the set of permutations on $\{1,\dots,d\}$. For a vector $x \in \R^d$ and a permutation $\tau \in \mathcal{T}$ we write $x_{\tau}$ for the vector $(x_{\tau(1)},\dots,x_{\tau(d)})$. For a semimartingale $X$ and $a \in \R$, we denote by $L_X^a$ its local time process at $a$,
$$L_X^a(t) = |X(t)-a| - |X_0 - a| - \int_0^t \text{sign}(X(s)-a)dX(s)$$
where sign$(x) = 1_{x > 0} - 1_{x \leq 0}$. We write $L_X$ for $L_X^0$.
\section{Dirichlet Form and Integration by Parts} \label{sec:Dirichlet_form}
Fix $d\geq 2$ and a nonempty open connected domain $E \subset \R^d$. For a permutation $\tau \in \mathcal{T}$ define the set
$$E_{\tau} := \{x \in E: x_{\tau(1)} > x_{\tau(2)} > \dots > x_{\tau(d)}\}.$$
Note that the sets $\{E_\tau\}_{\tau \in \mathcal{T}}$ are pairwise disjoint. 
We write $\bar E_\tau$ for the closure of $E_\tau$ under the subspace topology on $E$ and define $\partial E_\tau = \bar E_\tau \setminus E_\tau$. Then $E = \cup_{\tau \in \mathcal{T}} \bar E_\tau$.
We now define a certain class of functions on $E$ that will be suitable for the construction of our process in the next section. 

\begin{defn} 
	\begin{enumerate}[label = {(\arabic*)}]
		\item For a set $V$, a function $f:E \to V$ and a permutation $\tau \in \mathcal{T}$ we will write $f_\tau$ for $f|_{\bar E_\tau}$. We call the collection of functions $\{f_\tau\}_{\tau \in \mathcal{T}}$ the \emph{rank decomposition of $f$} and we have the identity $f = \sum_{\tau \in \mathcal{T}}f_\tau\I_{E_\tau}$ away from the common boundaries $\partial E_\tau \cap \partial E_{\tau'}$.
		\item For each $k \geq 1$ define the sets
		\begin{align*}C^k_{\mathcal{T}}(E;V) & := \{ f \in C^0(E;V):  \ f_\tau \in C^k(E_\tau;V) \text{ for every } \tau \in \mathcal{T}\},\\
		C^k_{\mathcal{T},c}(E;V)&:= C^k_{\mathcal{T}}(E;V) \cap C_c^0(E;V).
	 \end{align*}
	\end{enumerate}
\end{defn}
The set $C^k_{\mathcal{T}}(E;V)$ contains all functions $f$ that are both continuous on $E$ and $k$-times continuously differentiable on $E_{\tau}$ for each $\tau \in \mathcal{T}$. We also have the inclusions
\[C^k(E;V) \subset C^k_{\mathcal{T}}(E;V) \subset C^0(E;V); \qquad C^k_c(E;V) \subset C^k_{\mathcal{T},c}(E;V) \subset C_c^0(E;V)\]
for every $k \geq 1$.

We take two inputs $(c,p)$, which will play the role of a the volatility matrix and invariant density, respectively, of the process we construct. They are assumed to satisfy:
\begin{assum}\label{ass:inputs}  $c \in C^1_{\mathcal{T}}(E;
		\mathbb{S}^d_{++})$ and $p \in C_{\mathcal{T}}^1 (E;(0,\infty))$ are such that $cp$ is locally Lipschitz continuous; that is for every compact set $K \subset E$, $cp$ is Lipschitz continuous on $K$.
\end{assum}
We define the pre-Dirichlet form
\begin{equation} \label{eqn:Dirichlet_form}
	\mathcal{E}(u,v) := \frac{1}{2}\int_{E} \nabla u^\top c \nabla v p; \quad u,v \in 
\mathcal{D} := C_c^\infty(E).
\end{equation}
$\mathcal{E}$ is a symmetric, Markovian bilinear form. We refer the reader to \cite{fukushima2010dirichlet} for definitions of these notions and for an in-depth treatment of the theory of Dirichlet forms.
By \cite[Proposition~1.9]{baur2013construction} we have that $(\mathcal{E}, \mathcal{D})$ is closable on $L^2(E;\mu)$ where the  measure $\mu$ is defined as \[d\mu(x) = p(x)dx.\] Note that we do not require $\mu$ to be a finite measure at this stage. We denote the closure of $(\mathcal{E},\mathcal{D})$ by $(\mathcal{E},D(\mathcal{E}))$. The closure is a symmetric, strongly local and regular Dirichlet form on $L^2(E;\mu)$. For every $i \in \{1,\dots,d\}$ define  \[\diver c_i = \sum_{j=1}^d \partial_jc_{ij}, \qquad b = \frac{1}{2}\diver c+ \frac{1}{2}c\nabla \log p.\] Note that, by Assumption~\ref{ass:inputs}, both $p$ and the components of $c$ are differentiable outside of a Lebesgue null-set so that $\diver c$ and $b$ are well-defined up to a Lebesgue null-set. Define the operator
\begin{equation} \label{eqn:generator} Lf = 	b^\top \nabla f + \frac{1}{2}\tr(c \nabla^2 f); \quad f \in \mathcal{D}.
\end{equation}
We establish the following integration by parts formula.
\begin{thm}[IBP] \label{thm:IBP}
	Let $v \in C^1(E)$ and $\xi \in C^1_{\mathcal{T},c}(E;\R^d)$ be given. Then
	\begin{equation} \label{eqn:IBP} \frac{1}{2}\int_{E} \nabla v^\top c \xi p = -\frac{1}{2}\int_{E} v\diver(c\xi p).
	\end{equation}
	In particular when $\xi = \nabla u$ for some $u \in C^2_{\mathcal{T},c} (E)$, \eqref{eqn:IBP} becomes $\mathcal{E}(u,v) = (-Lu,v)_{L^2(E;\mu)}$.
\end{thm}
\begin{proof}  
	For every $\tau \in \mathcal{T}$ and $x \in \partial E_\tau$, denote by $\nu_{\tau}(x)$ an outward pointing normal vector at $x$ from $E_\tau$. Denote by $\partial \mathcal{T}^2$ the set of all pairs $\tau,\tau' \in \mathcal{T}$ such that $\tau \ne \tau'$ and $\partial E_\tau \cap \partial E_{\tau'}$ has co-dimension $d-1$. Then, for any pair $(\tau,\tau') \in \partial \mathcal{T}^2$ we have by anti-symmetry the relationship
		\begin{equation} \label{eqn:normal_vect_relation}
	\nu_{\tau}(x) = - \nu_{\tau'}(x); \quad \mathcal{H}^{d-1}\text{-a.e }x \in \partial E_{\tau} \cap \partial E_{\tau'}, 
	\end{equation} 
where $\mathcal{H}^{d-1}$ is the $d-1$ dimensional Hausdorff measure.
	Now we turn our attention to establishing the integration by parts formula. Using the fact that $v,c_\tau,p_\tau,\xi$ are continuous on $\bar E_\tau$, $\xi$ vanishes on a neighbourhood of $\partial E$ and, by virtue of Assumption~\ref{ass:inputs}, $c\xi p$ is Lipschitz continuous on $\bar E_\tau  \cap \supp(\xi)$  we can directly integrate by parts (see\ \cite[Corollary~9.66]{MR3726909}) to obtain 
	\begin{align}\int_{E} \nabla  v^\top c  \xi p  &= \sum_{\tau \in \mathcal{T}} \int_{E_\tau} \nabla v^\top c_\tau \xi p_\tau \nonumber \\
	& = \sum_{\tau \in \mathcal{T}} \int_{E_\tau} -v \diver (c_\tau \xi p_\tau) + \int_{\partial E_\tau} v \nu_\tau^\top c_\tau \xi p_\tau d\mathcal{H}^{d-1} \nonumber\\
	& = \int_{E} -v\diver(c\xi p) +  \sum_{(\tau,\tau') \in \partial\mathcal{T}^2} \int_{\partial E_\tau \cap \partial E_{\tau'}} v \nu_{\tau}^\top (c_\tau p_\tau - c_{\tau'}p_{\tau'})\xi d\mathcal{H}^{d-1} \label{eqn:IBP_boundary} \\
	& =  \int_{E} -v\diver(c\xi p). \nonumber
	\end{align}
 The second to last equality follows from \eqref{eqn:normal_vect_relation}, while the last equality follows since $c_\tau p_\tau = c_{\tau'}p_{\tau'}$ on $\partial E_{\tau} \cap \partial E_{\tau'}$. This proves the general integration by parts formula.
	The final claim in the statement of the theorem now follows from the fact that for a function $u \in C^2_{\mathcal{T},c}(E)$ we have $pLu = \frac{1}{2}\diver (c\nabla up)$ almost everywhere.
\end{proof}

Theorem \ref{thm:IBP} implies that the generator associated with $\mathcal{E}$ coincides with $L$ when acting on functions in $\mathcal{D}$. With some abuse of notation we set
$$(L,D(L)) = \text{generator of } (\mathcal{E},D(\mathcal{E})).$$
\section{Construction of the Process}\label{sec:construction}
To construct our process of interest we extend the state space.
Let $\hat E := E \cup \{\Theta\}$ be the one point compactification of $E$. By convention, we extend any real-valued function $f$ on $E$ to a function on $\hat E$ by setting $f(\Theta) = 0$. 
We will use the results of \cite{baur2013construction} to obtain a diffusion process, possessing the so-called \emph{$L^p$-strong Feller} property, corresponding to our Dirichlet form. 

For the reader's convenience we state the relevant part of that theorem.
\begin{thm}[{\citealp[Theorem~1.12]{baur2013construction}}] \label{thm:strong_feller} 
	There exists a filtered probability space  $(\Omega,\F,\{\F(t)\}_{t \geq 0}, \{\P_x\}_{x \in \hat E})$ that supports a Hunt process $\{X(t)\}_{t \geq 0}$
	with state space $E$ and cemetery state $\Theta$. The process has continuous paths on $[0,\infty)$ and yields a solution to the martingale problem for $(L,D(L))$; that is 
	\begin{equation} \label{eqn:MP} M^{[u]}(t) := u(X(t)) -  u(x) - \int_0^t Lu(X(s))ds, \quad t \geq 0
	\end{equation}is an $(\F(t))$-martingale under $\P_x$ for every $u \in D(L)$ and $x \in \hat E$. Furthermore, for every $d < p< \infty$ the transition semigroup $(P_t)_{t \geq 0}$ is $L^p
$-strong Feller; that is $P_tf \in C^0(E)$ for every $f \in L^p(E;\mu)$ and $t > 0$.
\end{thm} 
Note that if $\mu$ is a finite measure then $\mathcal{B}_b(E) \subset L^p(E;\mu)$ for each $r \geq 1$ so that, in this case, $X$ also possesses the (standard) strong Feller property. As a consequence of Theorem~\ref{thm:strong_feller} we obtain absolute continuity of the transition function corresponding to $X$. This was already shown in the proof of \cite[Theorem~2.8]{baur2013construction}, but since this fact plays a crucial role in proving the absence of triple collisions we reproduce the proof below.
\begin{cor}[Absolute continuity condition] \label{cor:abs_continuity}
	Let $\{p_t\}_{t > 0}$ denote the transition function corresponding to the process $X$ guaranteed by Theorem~\ref{thm:strong_feller}. Then $p_t(x,\cdot)$ is absolutely continuous with respect to $\mu$ for each $x \in E$ and $t > 0$.
\end{cor}
\begin{proof}
Let $\{F_k\}_{k=1}^\infty$ be an increasing sequence of sets such that  $E = \cup_{k=1}^\infty F_k$ and $\mu(F_k) < \infty$ for each $k$. Let $N \in \mathcal{B}(E)$ be such that $\mu(N) = 0$ and fix $t > 0$. Note that for each $k$ the function $f_k := 1_{N \cap F_k}$ is a member of $L^p(E;\mu)$ for all $1 \leq p \leq \infty$. Since $P_t$ is a symmetric operator on $L^2(E;\mu)$ it follows that for any $g \in L^2(E;\mu)$ and $k \in \N$ we have \[(g,P_tf_k)_{L^2(E;\mu)} = (P_tg,f_k)_{L^2(E;\mu)} = 0.\] Hence, we obtain that $p_t(x,N \cap F_k) = 0$ for $\mu$-a.e $x \in E$. However, by Theorem~\ref{thm:strong_feller} the function $p_t(\cdot, N \cap F_k) = P_tf_k$ is continuous so it follows that $p_t(x,N\cap F_k) = 0$ for every $x \in E$. By monotone convergence we obtain that $p_t(x,N) = \lim_{k \to \infty} p_t(x,N \cap F_k) = 0$ which completes the proof.
\end{proof}
We will now show that the process $X$ constructed from Theorem~\ref{thm:strong_feller} is a semimartingale prior to its explosion time and obtain its semimartingale decomposition. For any open set $G \subset \hat E$ set
$$\mathcal{D}_G = \{u \in \mathcal{D}: \text{Supp}(u) \subseteq G\}.$$
We now state \cite[Theorem~6.3]{fukushima1999semi}, which we will employ to establish the semimartingale property.
\begin{thm}[{\citealp[Theorem~6.3]{fukushima1999semi}}] \label{thm:semimrt_character} The following conditions are equivalent for $u \in D(\mathcal{E})$.
	\begin{enumerate}[label = ({\roman*})] \itemsep0em
		\item $u(X(t)) - u(X(0))$ is a semimartingale,
		\item For any relatively compact open set $G \subset \hat E$, there exists a positive constant $C_G$ such that 
		\begin{equation} \label{eqn:semimrt_equiv}
		|\mathcal{E}(u,v)| \leq C_G\|v\|_\infty, \quad \forall v \in \mathcal{D}_G,
		\end{equation}
	\end{enumerate}
\end{thm}
Since $X$ may, in general, explode in finite time the semimartingale property of $X$ will only hold prior to this random time. To precisely state the result let $\{U_n\}_{n \in \N}$ be a sequence of open relatively compact sets in $E$ such that $\cup_n U_n = E$ and $\bar U_n \subset U_{n+1}$. We define
\begin{align*}
	\zeta_n & := \inf\{t \geq 0: X(t) \in U_n^c\}, \quad n \in \N,\\
	\zeta & := \inf\{t \geq 0: X(t) = \Theta\}.
\end{align*}
It is clear that $\lim_{n \to \infty} \zeta_n = \zeta$ and we denote by $X^{\zeta_n}_t := X_{t \land \zeta_n}$ the stopped process.

\begin{prop} \label{prop:X_semi}
	For each $n \in \N$, $X^{\zeta_n}$ is a semimartingale. Moreover for each $x \in E$ under $\P_{x}$, $X$ satisfies the SDE
	\begin{equation} \label{eqn:X_semimart_decomp}
		dX(t) = b(X(t))dt + c^{1/2}(X(t))dW(t); \quad X(0) = x
	\end{equation}
on the random interval $[0,\zeta)$, where $c^{1/2}(x)$ is a matrix square root of $c(x)$ and $W$ is a $d$-dimensional Brownian motion.
\end{prop}
\begin{proof}
	Since the semimartingale property of a vector valued process is defined coordinate-wise we prove that $X_i^{\zeta_n}$ is a semimartingale for every $i=1,\dots,d$ and $n \in \N$. Fix such indices $i,n$ and let $\phi_n: E \to [0,\infty)$ be a smooth cutoff function such that $\phi_n = 1$ on $U_n$ and $\phi_n = 0$ on $E\setminus U_{n+1}$. Let \[C_n := \sup_{x \in E} |\sum_{|\alpha|\leq 2} \partial^\alpha \phi_n(x)|\] where $\alpha  \in \N^d$ is a multi-index and $|\alpha| = \sum_{i=1}^d \alpha_i$. Note that $C_n < \infty$. Next define $u_i^n(x) :=x_i\phi_n(x)$. It is clear that $u_i^n(x) = x_i$ on $U_n$ and $u_i^n\in D(\mathcal{E})$. Since $X_i^{\zeta_n}(t) = X_i(t)$ for $t \in [0,\zeta_n)$ to prove the claim, by Theorem~\ref{thm:semimrt_character}, it suffices to verify  \eqref{eqn:semimrt_equiv} for $u_i^n$. To this end fix  a relatively compact open set $G$ and $v \in \mathcal{D}_G$. Then using \eqref{eqn:IBP} we see that 
	$$|\mathcal{E}(u_i^n,v)| = \left| \int_{E} vLu_i^np \right| \leq \|v\|_\infty C_n\int_{G \cap U_{n+1}} (|\diver c_i| + |c\nabla \log p_i|)p.$$
	It follows that \eqref{eqn:semimrt_equiv} holds with $C_G = C_n\int_{G \cap U_{n+1}} (|\diver c_i| + |(c\nabla \log p)_i|)p$.

To prove the semimartingale decomposition formula \eqref{eqn:X_semimart_decomp} fix $x \in E$. Choose $n$ large enough so that $x \in U_n$. Then from the semimartingale property of $X^{\zeta_n}$ and Theorem~\ref{thm:strong_feller} we have for $t \in [0,\zeta_n)$ that 
$$dX_i(t) = du_i^n(X(t)) = b(X_i(t))dt + c^{1/2}(X_i(t))dW(t)$$
for some Brownian motion $W$.
Sending $n \to \infty$ we see that \eqref{eqn:X_semimart_decomp} holds for every $x \in E$ and for every $t \in [0,\zeta)$.
\end{proof}

\section{The Study of Collisions} \label{sec:collisions}
In this section we study particle collisions.
First we establish that collisions only occur at a nullset of time points.

\begin{lem} \label{lem:coliision time_set}
	The set $\{t \in [0,\zeta): X_i(t) = X_j(t) \text{ for some } i \ne j\}$ is $\P_x$-a.s a Lebesgue null-set for every $x \in E$.
\end{lem}
\begin{proof} Fix indices $i \ne j$. Note that $d\langle X_i-X_j\rangle(t) = (c_{ii}(X(t)) - 2c_{ij}(X(t)) + c_{jj}(X(t)))dt$. The occupation density formula then yields
	$$\int_0^t f(X_i(s) - X_j(s)))(c_{ii}(X(s)) - 2c_{ij}(X(s)) + c_{jj}(X(s)))ds = \int_{\R} f(a)L^a_{X_i - X_j}(t)da$$
	for every bounded measurable function $f$. Taking the function $f(a) = 1_{\{0\}}(a)$ we obtain 
	\begin{equation} \label{eqn:occ_dens}
	\int_0^t 1_{\{X_i(s) = X_j(s)\}}(c_{ii}(X(s)) - 2c_{ij}(X(s)) + c_{jj}(X(s)))ds = 0.
	\end{equation}
	Since $c(x) \in \mathbb{S}^d_{++}$ we have that $c_{ii}(x) - 2c_{ij}(x) + c_{jj}(x) > 0$ for every $x \in E$.
	Hence, taking expectation in \eqref{eqn:occ_dens} we conclude that $\P_x(X_i(t) = X_j(t)) = 0$ for every $x \in E$ and a.e.~$t
	\in [0,\zeta)$. Noting that
	$$\{t \in [0,\zeta): X_i(t) = X_j(t) \text{ for some } i \ne j\} \subset \bigcup_{i \ne j} \left\{t \in [0,\zeta): X_i(t) = X_j(t)\right\},$$ completes the proof.
\end{proof} 

Now we move on to the more subtle question of triple collisions. To establish this property we will need to introduce the notion of capacity. Set $\mathcal{E}_1(u,v) = (u,v)_{L^2(E;\mu)} + \mathcal{E}(u,v)$ for $u,v \in D(\mathcal{E})$. For any nonempty open set $U \subset E$ let 
$$\mathcal{L}_U:= \{u \in D(\mathcal{E}): u \geq 1 \text{ a.e on } U\}$$ and define
$$ \capac(U) = \inf_{u \in \mathcal{L}_U}\mathcal{E}_1(u,u).$$ For any nonempty set $A \subset E$ set $$\capac(A) = \inf_{U \text{ open, } A \subset U} \capac(U).$$
By convention we set $\capac(\emptyset) = 0$. 
We will first show that the set 
\begin{equation} \label{eqn:triple_collision}
B := \{x \in E: x_i = x_j = x_k \text{ for some distinct indices } i,j,k\}
\end{equation} has zero capacity. This, together with the absolute continuity criterion Corollary~\ref{cor:abs_continuity} will establish that the set $B$ is polar for $X$. 

	\begin{lem}\label{lem:cap0}
		$\mathrm{Cap}(B) = 0$ where $B$ is given by \eqref{eqn:triple_collision}.
	\end{lem}
	\begin{proof}
			For distinct indices $i,j,k \in \{1,\dots,d\}$ define the sets
		\begin{equation} \label{eqn:triple_collision_set}
		B^{ijk} := \{x \in E:  x_i = x_j = x_k \}.
		\end{equation} 
	Since $B$ is contained in a finite union of such sets it suffices to show that $\mathrm{Cap}(B^{ijk}) = 0$ for any distinct indices $i,j,k$. Fix such a triple of indices and assume that $B^{ijk} \ne \emptyset$, as the result trivially holds if it is empty. Note then that $B^{ijk}_n \uparrow B^{ijk}$ as $n \to \infty$ where
		$$B^{ijk}_n := \{x \in E: x_i = x_j =x_k, \ \mathrm{dist}(x,\partial E) > 1/n\} \cap Q_n$$
		and $Q_n = \{x \in \R^d: |x_l| < n, \ \forall l \in \{1,\dots,d\}\}$. Here the boundary $\partial E$ is understood to be in the topology on $\R^d$.
		If $\partial E = \emptyset$ then set dist$(x,\partial E) = \infty$ for every $x \in E$.
		By properties of a Choquet capacity \cite[Theorem~2.1.1]{fukushima2010dirichlet} we have Cap($B^{ijk}$) = $\sup_n$Cap($B^{ijk}_n$), so it suffices to prove that Cap($B^{ijk}_n$) $=0$ for every $n$.
		
		 Now fix $n$ and $\epsilon$ such that $0 < \epsilon(1+\epsilon) < 1/n$. Define 
		$$U^\epsilon_n = \{x \in E \ | \   |x_l - x_m| \in (-\epsilon,\epsilon)  \text{ for } l,m \in \{i,j,k\}, \ \mathrm{dist}(x,\partial E)  > 1/n - \epsilon\}  \cap Q_{n+\epsilon}.$$
		Each $U^\epsilon_n$ is open in $E$ and contains  $B^{ijk}_n$. Define the function
		$f: E \to \R$ given by 
		$$f(x) = \max\left\{\max_{\substack{ l,m \in \{i,j,k\} \\ l \ne m}}\left\{|x_l - x_m| \right\}, 1/n- \mathrm{dist}(x,\partial E), \max_{l \in \{1,\dots,d\}}|x_l| - n \right\}$$ and the
		logarithmic cutoff function $\eta: E \to \R$ via $$\eta(x) = -  \frac{\log \(\frac{f(x)}{\epsilon(1+\epsilon)}\)}{\log(1+\epsilon)} \land 1 \lor 0.$$
		
		It is clear that $\eta=1$ on $U^\epsilon_n$ and $\eta = 0$ on $E \setminus U^{\epsilon(1+\epsilon)}_n$. Additionally we have $|\nabla \eta(x)| \leq  (\log (1+\epsilon))^{-1} f(x)^{-1}$ for almost every $x \in U_n^{\epsilon(1+\epsilon)}\setminus U_n^{\epsilon}$. Since $\tr(c)$ and $p$ are both bounded on the set $\{\mathrm{dist}(x,\partial E) > 2/n\} \cap Q_{2n}$ by some constant $C_n$ we see that 
		\begin{align} \mathrm{Cap}(U_n^\epsilon) \leq \mathcal{E}_1(\eta,\eta) & = \int_{E} \eta^2 p + \frac{1}{2}\int_{E} \nabla \eta^\top c \nabla \eta p  \nonumber \\
		& \leq C_n \lambda^{d}(U^{\epsilon(1+\epsilon)}_n) + \frac{d}{2}C_n^2(\log (1+\epsilon))^{-2} \int_{U^{\epsilon(1+\epsilon)}_n\setminus U^\epsilon_n} f^{-2} \label{eqn:capacity_cal},
		\end{align}
	where $\lambda^d$ is the Lebesgue measure on $\R^d$.
	Note that $f^{-2}(x) \leq |x_i-x_j|^{-2}$ for $x \in U^{\epsilon(1+\epsilon)}_n\setminus U^\epsilon_n$ and we have
$$ U^{\epsilon(1+\epsilon)}_n\setminus U^\epsilon_n \subseteq Q_{2n} \cap \{x: \epsilon \leq |x_i - x_j| \leq \epsilon(1+\epsilon)\} \cap \{x: -\epsilon(1+\epsilon) + x_i \leq x_k \leq \epsilon(1+\epsilon) + x_i \}.$$
Hence we obtain the estimate 
\begin{align*} \int_{U^{\epsilon(1+\epsilon)}_n\setminus U^\epsilon_n} f^{-2} & \leq (4n)^{d-3} \int_{-2n}^{2n}\int_{ -\epsilon(1+\epsilon) + x_i}^{ \epsilon(1+\epsilon)+x_i}\int_{\{\epsilon \leq |x_i - x_j| \leq \epsilon(1+\epsilon)\}} |x_i - x_j|^{-2}dx_jdx_kdx_i \\
	& = 2^{2d-3}n^{d-2} \epsilon(1+\epsilon)\(\frac{1}{\epsilon} - \frac{1}{\epsilon(1+\epsilon)}\) = 2^{2d-3}n^{d-2}\epsilon.
	\end{align*}
Plugging this into \eqref{eqn:capacity_cal} yields the bound
\begin{equation} \label{eqn:capacity_bound}
\mathrm{Cap}(U_n^\epsilon) \leq \tilde C_n\( \frac{\epsilon}{\log (1+\epsilon)^2} + \lambda^{d}(U_n^{\epsilon(1+\epsilon)})\)
\end{equation}
for some constant $\tilde C_n > 0$. Since the right hand side of \eqref{eqn:capacity_bound} converges to $0$ as $\epsilon \downarrow 0$ we obtain
$$\mathrm{Cap}(B^{ijk}_n) \leq \lim_{\epsilon \downarrow 0} \text{Cap}(U^\epsilon_n) = 0.$$
This completes the proof.
\end{proof}

We are now ready to establish that triple collisions do not occur.

\begin{thm} [No Triple Collisions] \label{thm:triple_collision} We have
	$$\P_x(X(t) \in B\text{ for some } t > 0) = 0$$
	for every $x \in E$, where $B$ is given by \eqref{eqn:triple_collision}.
\end{thm}
\begin{proof}
	Since all compact sets in $E$ clearly have finite capacity it follows from  \cite[Theorem~4.2.1~(ii)]{fukushima2010dirichlet} that a set has zero capacity if and only if it is exceptional (in the sense of \cite[Page~152]{fukushima2010dirichlet}). Thanks to the absolute continuity condition Corollary~\ref{cor:abs_continuity}, it follows from \cite[Theorem~4.2.4]{fukushima2010dirichlet} that all exceptional sets are polar. Thus the result follows from Lemma~\ref{lem:cap0}.
\end{proof}

Using the above theorem we can obtain a simplified expression for the rank-process of $X$. For $x \in \R^d$, denote by $x_{()}$ the \emph{rank-vector} of $x$ defined by the conditions $x_{(1)} \geq x_{(2)} \geq \dots \geq x_{(d)}$ and $\{x_i: i = 1,\dots,d\} = \{x_{(k)}: k=1,\dots,d\}$. Next for each $k \in \{1,\dots,d\}$ define the \emph{rank identifying functions} $r_k:\R^d \to \{1,\dots,d\}$ via $r_k(x) = i$ if $x_{(k)} = x_i$ with ties broken by lexicographical ordering. Then, as a consequence of Lemma~\ref{lem:coliision time_set} and Theorem~\ref{thm:triple_collision}, we obtain that the assumptions of  \cite[Corollary~2.6]{banner2008local} are satisfied, which yields a simplified formula for the dynamics of the ranked market weights:
\begin{cor}
	 For $t \in [0,\zeta)$ the ranked semimartingale process $X_{()}$ has dynamics
\begin{equation} \label{eqn:rank_based_dynamics}
	dX_{(k)}(t) = \sum_{i=1}^d 1_{\{r_k(X(t)) = i\}}dX_i(t) + \frac{1}{4}dL_{X_{(k)} - X_{(k+1)}}(t) - \frac{1}{4}dL_{X_{(k-1)} -
		X_{(k)}}(t)
\end{equation}
for every $k =1,\dots,d$,
with the convention that $L_{X_{(0)} - X_{(1)}} = L_{X_{(d)}-X_{(d+1)}} = 0$.
\end{cor}

\section{Ergodicity and Occupation Times} \label{sec:ergodicty}
Next we turn to the question of ergodicity for both $X$ and the rank process $X_{()}$ as well as investigating how much time, asymptotically, $X$ spends in each region $E_{\tau}$.  The assumption $p > 0$ on $E$ implies that the Dirichlet form $\mathcal{E}$ is irreducible so it follows that $\mathcal{E}$ is ergodic if and only if it is recurrent. For more details on the notions of irreducibility, recurrence and ergodicity of Dirichlet forms we refer the reader to \cite[Section~1.6]{fukushima2010dirichlet}. It is clear that for the Markov process $X$ to be recurrent it cannot explode in finite time, which motivates the following assumption.
\begin{assum}
	\label{ass:non_explosion}
	$\P_x(\zeta < \infty) = 0$ for every $x \in E$.	
\end{assum}
Assumption~\ref{ass:non_explosion} is equivalent to $\mathcal{E}$ being conservative  \cite[Exercise~4.5.1]{fukushima2010dirichlet}. As such, there are test function methods available (see \cite[Theorem~1.6.6]{fukushima2010dirichlet}) to establish whether or not Assumption~\ref{ass:non_explosion} holds for specific choices of domain $E$ and inputs $(c,p)$. In Proposition~\ref{prop:suff_cond} below we establish a sufficient condition on the inputs $(c,p)$ for this assumption to hold when the domain is the simplex. 

It turns out that if $\mu$ is a finite measure then the conservativity property Assumption~\ref{ass:non_explosion} is both necessary and sufficient to establish recurrence. Indeed, it is always the case  that recurrence implies conservativity \cite[Lemma~1.6.5]{fukushima2010dirichlet}, while the fact that conservativity implies recurrence in this context follows by virtue of \cite[Theorem~6.3.2]{fukushima2010dirichlet} with the choice $\phi \equiv 1$ (which is admissible by the finiteness of $\mu$) in the notation of that theorem. As a consequence we obtain the following Birkhoff ergodic theorem:
\begin{prop}[{\citealp[Theorem 6.3.3(iii)]{fukushima2010dirichlet}}] \label{prop:ergodic}
	Let Assumption~\ref{ass:non_explosion} hold and assume $\mu$ is a finite measure. Then 
	\begin{equation} \label{eqn:ergodic}
		\lim_{T\to \infty} \frac{1}{T}\int_0^T f(X(t))dt = \frac{1}{\mu(E)} \int_{E}fp; \quad \P_x\text{-a.s}
	\end{equation}
	for every $f \in L^1(E;\mu)$ and $x \in E$.
\end{prop}
For the remainder of this section we assume that $\mu$ is a finite measure and that Assumption~\ref{ass:non_explosion} holds. By the ergodic property \eqref{eqn:ergodic} the long-term occupation times of the sets $E_{\tau}$ are given by

\begin{equation}\label{eqn:occupation_time}
	\theta_{\tau} := \lim_{T \to \infty} \frac{1}{T}\int_0^T 1_{E_\tau}(X(t))dt = \frac{\mu(E_\tau)}{\mu(E)}
\end{equation} for every $\tau \in \mathcal{T}$. The asymptotic average occupation time that the coordinate $X_i$ spends in the $k^{\text{th}}$ rank is given by
\begin{equation}\label{eqn:rank_occupation_time} \theta_{k,i} := \lim_{T \to \infty} \frac{1}{T}\int_0^T 1_{\{r_k(X(t)) = 1\}}dt = \sum_{\tau \in \mathcal{T}, \tau(k) = i} \lim_{T \to \infty} \frac{1}{T}\int_0^T 1_{E_\tau}(X(t))dt = \sum_{\tau \in \mathcal{T}, \tau(k) = i}\theta_{\tau}.
\end{equation}
for $i,k=1,\dots,d$. Additionally, the ergodicity of $X$ implies a Birkhoff ergodic theorem  for $X_{()}$. Set $$E_{\geq} = \{y \in \R^d: y = x_{()} \text{ for some } x \in E\}$$ to denote the set of ranks of vectors in $E$. Note that $\tau (\bar E_\tau) \subseteq E_{\geq}$ and $E_{\geq} = \bigcup_{\tau} \tau(\bar E_\tau)$ where $\tau(\bar E_\tau) := \{x_\tau: x \in E_{\tau}\}$. Define $q:E_{\geq} \to (0,\infty)$ via
\begin{equation} \label{eqn:ranked_density}
	q(y) = \sum_{\tau \in \mathcal{T}} p_\tau(y_{\tau^{-1}}),
\end{equation}
where $p_\tau(y_{\tau^{-1}})$ is defined to be zero if $y_{\tau^{-1}} \notin E_{\tau}$. Intuitively, $q$ will act as an invariant density for $X_{()}$. Indeed we obtain the following corollary, which shows that $X_{()}$ satisfies Birkhoff's ergodic theorem with (unnormalized) invariant density $q$:
\begin{cor}
Let $d\nu(y) = q(y)dy$ be a measure on $E_{\geq}$.
Then for any $f \in L^1(E_{\geq};\nu)$ we have
\[\lim_{T \to \infty} \frac{1}{T} \int_0^T f(X_{()}(t))dt = \frac{1}{\nu(E_{\geq})} \int_{E_{\geq}} f(y)q(y)dy.\]
\end{cor}
\begin{proof}
	Fix $f \in L^1(E_{\geq};\nu)$. Note
 that $\tilde f \in L^1(E;\mu)$ where $\tilde f(x):= f(x_{()})$. Hence we have by Proposition~\ref{prop:ergodic} that  

$$\lim_{T \to \infty} \frac{1}{T} \int_0^T f(X_{()}(t))dt = \frac{1}{\mu(E)} \int_{E} f(x_{()})p(x)dx = \frac{1}{\nu(E_{\geq})} \int_{E_{\geq}} f(y)q(y)dy,$$
where in the last equality we used a change of variables and the fact that $\mu(E) = \nu(E_{\geq})$. 
\end{proof}
We end this section with a few examples. 
\begin{eg}[Pure Rank-Based Models]
	Suppose the domain $E$ is \emph{rank-symmetric}; that is $\tau(E_\tau) = E_{\geq}$ for every  $\tau \in \mathcal{T}$.  Take as inputs  $\kappa \in C^1(E_{\geq};\mathbb{S}^d_{++}) \cap C^0(\bar E_{\geq};\mathbb{S}^d_{++})$ and $q \in C^1(E_{\geq};(0,\infty)) \cap C^0(\bar E_{\geq};(0,\infty))$, where as before the closure $\bar E_{\geq}$ is understood to be with respect to the subspace topology on $E$. The inputs $\kappa$ and $q$ will serve as the covariation matrix and invariant density for $X_{()}$, the ranked process of $X$. To construct $X$ we first extend $\kappa$ and $q$ to all of $E$ by symmetrization. That is, we define $c: E \to \mathbb{S}^d_{++}$ and $p: E \to (0,\infty)$ via 
	\begin{align*}
		c_{ij}(x) & = \sum_{k,l=1}^d 1_{\{r_k(x) = i\}}1_{\{r_l(x) = j\}}\kappa_{kl}(x_{()}); \quad i,j=1,\dots,d,\\
		p(x) & = \frac{1}{d!}q(x_{()}).
	\end{align*}
It is clear that $(c,p)$ satisfy Assumption~\ref{ass:inputs} so that the results of the previous section apply. In particular the rank process of $X$ is given by
\begin{equation} \label{eqn:pure_rank_dynamics}
	dX_{(k)}(t) = \beta_k(X_{()}(t))dt + \kappa_k^{1/2}(X_{()}(t))dW(t) +  \frac{1}{4}dL_{X_{(k)} - X_{(k+1)}} - \frac{1}{4}dL_{{X_{(k-1)}} - X_{(k)}}; \quad t < \zeta
\end{equation} for $k =1,\dots,d$ where
$\beta:E_{\geq} \to \R^d$ is given by $\beta(y) = \frac{1}{2}\diver \kappa(y) + \frac{1}{2}\kappa(y)\nabla \log q(y)$. Since the dynamics \eqref{eqn:pure_rank_dynamics} depend only on the rank vector process $X_{()}$ we call such a model a \emph{pure rank-based model}. 

Now we enforce Assumption~\ref{ass:non_explosion} and additionally assume that $\mu(E) < \infty$. As a consequence of Proposition~\ref{prop:ergodic} we obtain that $X$ is ergodic with (unnormalized) density $p$ and $X_{()}$ is ergodic with (unnormalized) density $q$. Moreover, in these pure rank-based models each coordinate $X_i$ asymptotically spends the same amount of time occupying each rank. Indeed, from \eqref{eqn:occupation_time}, we see that the asymptotic occupation time  of the set $E_{\tau}$ by $X$ is given by $\theta_{\tau} =  \mu(E_\tau)/\mu(E) = 1/d!$ in this setting. Since for any fixed $i,k \in \{1,\dots,d\}$ there are $(d-1)!$ permutations $\tau$ for which $\tau(k) = i$ it follows from \eqref{eqn:rank_occupation_time} that $\theta_{k,i} = 1/d$. 
\end{eg}

\begin{eg}[Competing Brownian Particle Systems; Common Volatility] \label{eg:Brownian_particle}
Let $E = \R^d$ and suppose $\sigma > 0$ and $g \in \R^d$ are given. Define $c \equiv \sigma^2 I_{d \times d}$ and $p(x) = \exp(\frac{2g}{\sigma^2} ^\top x_{()})$. Then Assumption~\ref{ass:inputs} is satisfied so we see that 
$$dX_i(t) = \sum_{k=1}^d g_k1_{\{r_k(X(t) = i)\}}dt + \sigma dW_i(t); \quad i=1,\dots,d.$$
This recovers the class of competing Brownian particle systems with common volatility discussed in the introduction. The triple collision and ergodic properties of the more general process 
\[Y_i(t) = \sum_{k=1}^d g_k1_{\{r_k(Y(t)) = i\}}dt + \sum_{k=1}^d \sigma_k1_{\{r_k(Y(t)) = i\}}dW_i(t); \quad i=1,\dots,d\]
for some constants $\sigma_k > 0$ has received a lot of recent attention in the literature \cite{MR3055258,MR2473654,MR3325099,ichiba2011hybrid}. Assumption~\ref{ass:inputs} is not satisfied for the volatility matrix $c(x) = \sum_\tau \sigma^2_\tau 1_{E_\tau}(x) I_{d \times d} $ due to its discontinuity at the common boundary $\partial E_\tau \cap \partial E_{\tau'}$ for distinct permutations $\tau,\tau'$. As such, we are unable to handle this specification outside the common volatility case. A further discussion of the limitations of our approach is carried out in Section~\ref{sec:conclusion}.
\end{eg}

\section{Stochastic Portfolio Theory} \label{sec:SPT}

 In this section we show how the theory developed in the previous sections relates to, and extends, the equity models previously considered in the literature.

  Suppose we are given the \emph{market capitalizations processes} of a collection of stocks; that is, $S = (S_1,\dots,S_d)$ represent the capitalizations of $d$ stocks. Then the process $X = (X_1,\dots,X_d)$ given by $X_i = S_i/(S_1+\dots+S_d)$ is called the \emph{market weight process}. If none of the stock capitalizations vanish then the process $X$ takes values in the open simplex,
\begin{equation} \label{eqn:simplex}
\Delta^{d-1}_+ := \left\{x \in (0,1)^d: \sum_{i=1}^d x_i = 1\right\}.
\end{equation}

As defined in \eqref{eqn:simplex}, $\Delta^{d-1}_+$  is a $d-1$ dimensional subset of $\R^d$. However, it can be identified with an open set $E \subset \R^{d-1}$ via the transformation \[(x^1,\dots,x^{d-1}) \mapsto (x^1,\dots,x^{d-1},1-\sum_{i=1}^{d-1}x^i).\] 
For ease of notation and consistency with the SPT literature we work with the set $\Delta^{d-1}_+$ rather than directly working with the set $E$. Since $\Delta^{d-1}_+$ can be viewed as a differentiable manifold where the vector fields $\partial_i - \partial_d$, $i=1,\dots,d-1$ span the tangent space at each point of the simplex, the formulas in the sequel involving derivatives of functions on $\Delta^{d-1}_+$ are unambiguous. Additionally in this section we will, by an abuse of notation, write $\mathbb{S}^d_{++}$ for the set off all symmetric matrices that satisfy the positive-definite property on the tangent space of $\Delta^{d-1}_+$; that is a symmetric $d \times d$ dimensional matrix $A$ is in $\mathbb{S}^d_{++}$ if it is positive semi-definite and $\text{Ker}(A) = \text{span}(\boldsymbol{1})$. See \cite{fernholz2002stochastic,itkin2020robust} for a more detailed discussion of these conventions.

As mentioned above the rank-process $X_{()}$ of a market weight process has empirically observed stability properties. The procedure of the previous sections can be employed to construct such processes. Let $c:\Delta^{d-1}_+ \to \mathbb{S}^d_{++}$ and $p: \Delta^{d-1}_+ \to (0,\infty)$ satisfying Assumption~\ref{ass:inputs} be given. Then we obtain a market weight process $X$ corresponding to the (pre-)Dirichlet form given by \eqref{eqn:Dirichlet_form}, with dynamics \eqref{eqn:X_semimart_decomp}. Additionally, there are no triple collisions and the dynamics of $X_{()}$ are given by \eqref{eqn:rank_based_dynamics}.  If $\int_{\Delta^{d-1}_+}p = 1$ and Assumption~\ref{ass:non_explosion} holds then we obtain that $X$ is ergodic with density $p$ and $X_{()}$ is ergodic with density $q$.

The approach here is to directly construct and work with the market weight process, rather than first constructing the stock capitalization process $S$. We made this choice as the results regarding ergodicity from Section~\ref{sec:ergodicty} are directly applicable to $X$, but not to $S$; indeed the stock capitalization process will typically not be ergodic. Once $X$ is constructed, however, one can always fashion an equity market that induces $X$ as its market weight process. Indeed, given a process $X$ on the simplex and any strictly positive process $\Sigma$ we can define a stock capitalization process $S$, by setting $S_i = X_i\Sigma$ for every $i=1,\dots,d$, which induces the market weight process $X$.

  Before considering examples we develop a sufficient condition on the inputs $(c,p)$ so that Assumption~\ref{ass:non_explosion} holds: 

\begin{prop} \label{prop:suff_cond}
	Suppose that \begin{equation} \label{eqn:suff_cond}
		\int_{\Delta^{d-1}_{+,\tau}}\frac{c_{\tau(d)\tau(d)}(x)}{x_{\tau(d)}^2}p(x)dx < \infty, \quad \tau \in \mathcal{T},
	\end{equation}
	 where $\Delta^{d-1}_{+,\tau} = \{x \in \Delta^{d-1}_+: x_{\tau(1)} > x_{\tau(2)} > \dots > x_{\tau(d)}\}$. Then Assumption~\ref{ass:non_explosion} holds.
\end{prop}
\begin{proof}
	We show that this condition directly implies $\mathcal{E}$ is recurrent, which in turn implies Assumption~\ref{ass:non_explosion}.
By \cite[Theorem~1.6.2]{fukushima2010dirichlet} this is equivalent to finding a sequence $\{u_n\}_{n \in \N} \subset D(\mathcal{E})$ such that $\lim_{n \to \infty} u_n = 1$ a.e and $\lim_{n \to \infty} \mathcal{E}(u_n,u_n)$ $= 0$. Set $$u_n(x) := n(x_{(d)} -1/n) \land 1 \lor 0$$ for $x \in \Delta^{d-1}_+$. It is clear that $u_n$ is $n$-Lipschitz and hence differentiable almost everywhere. Moreover, $u_n(x) = 0$ for $x \in \{x_{(d)} \leq 1/n\}$ so it is clear that $u_n \in D(\mathcal{E})$, as it is compactly supported and Lipschitz. We also see that $u_n(x)= 1$  for $x \in  \{x_{(d)} \geq 2/n\}$ from which it follows that $u_n(x) \to 1$ and $\nabla u_n(x) \to 0$ for a.e.~$x$.

Now fix $\tau \in \mathcal{T}$ and $x \in \Delta^{d-1}_{+,\tau}$. Since $u_n$ only depends on $x_{(d)}$ we see that $\partial_i u_n(x) = 0$ for all $i \ne \tau(d)$. Moreover, by the previous observation if $x_{\tau(d)} <1/n$ or $x_{\tau(d)} > 2/n$ then $\partial_{\tau(d)} u_n(x) = 0$ as well. For $x_{\tau(d)} \in (1/n,2/n)$ we estimate that
$\partial_{\tau(d)} u_n(x)  = n \leq \frac{2}{x_{\tau(d)}}$. Thus we obtain that
$$\mathcal{E}(u_n,u_n) =\frac{1}{2}  \int_{\Delta^{d-1}_+} \nabla u_n^\top c \nabla u_n p = \frac{1}{2}\sum_{\tau \in \mathcal{T}} \int_{\Delta^{d-1}_{+,\tau}} \partial_{\tau(d)}u_n(x)^2 c_{\tau(d)\tau(d)}(x)p(x)dx.$$
We have the estimate
$$\partial_{\tau(d)}u_n(x)^2 c_{\tau(d)\tau(d)}(x)p(x) \leq \frac{4c_{\tau(d)\tau(d)}(x)}{x_{\tau(d)}^2}p(x)$$
for every $\tau \in \mathcal{T}$,
which is integrable by assumption. Using the dominated convergence theorem we conclude that $\lim_{n \to \infty} \mathcal{E}(u_n,u_n) = 0$, completing the proof.
\end{proof}

\subsection{Atlas and Volatility-Stabilized Markets, Name-Based Drift} \label{eg:name_based_drift}
	
Let $\beta \geq 0$ and $\alpha,\sigma \in (0,\infty)^d$ be given.
For $x \in \Delta^{d-1}_+$ define
\begin{align*}p(x) & = 
\begin{dcases}
	Z^{-1}\(\sum_{i=1}^d\frac{x_i^{2\beta}}{\sigma_i^2}\)^{\frac{2(1+(d-1)\beta)-d}{2\beta}}\prod_{i=1}^d x_i^{\frac{2\alpha_i}{\sigma_i^2}+2\beta-1} & \beta \ne 0, \\
	Z^{-1}\prod_{i=1}^d x_i^{\frac{2\alpha_i}{\sigma_i^2}-1} & \beta = 0,
\end{dcases}\\
c_{ij}(x) & = -x_ix_j\( \sigma_i^2x_i^{1-2\beta}+ \sigma_j^2x_j^{1-2\beta}- \sum_{k=1}^d \sigma_k^2 x_k^{2(1-\beta)}\), && i,j=1,\dots,d, \ i \ne j
\end{align*}
and $c_{ii}(x) = \sum_{j \ne i} c_{ij}(x)$ for $i =1,\dots,d$, where $Z$ is a normalizing constant chosen so that $\int_{\Delta^{d-1}_+} p = 1$. It is easy to see that Assumption~\ref{ass:inputs} is satisfied for this choice of $(c,p)$. Moreover a direct calculation shows that \eqref{eqn:suff_cond} holds, so by Proposition~\ref{prop:suff_cond} it follows that Assumption~\ref{ass:non_explosion} holds for  the process $X$ corresponding to the (pre-)Dirichlet form defined by \eqref{eqn:Dirichlet_form}. The dynamics of $X$ in this case are given by
 \begin{equation} \label{eqn:vol_stab_market_dynam}
	\begin{split}
		\frac{dX_i}{X_i} = \((\alpha_i+\frac{\sigma_i^2}{2})X_i^{-2\beta} - \sigma_i^2 X_i^{1-2\beta} - \sum_{j=1}^d (\alpha_j+\frac{\sigma_j^2}{2})X_j^{1-2\beta} + \sum_{j=1}^d \sigma_j^2 X_j^{2(1-\beta)}\)dt \\
		+   \sigma_iX_i^{-\beta}(1-X_i)dW_i + \sum_{j \ne i} \sigma_jX_j^{1-\beta}dW_j
	\end{split}
\end{equation}
where we omitted the time argument for notational clarity. Since no triple collisions occur in this model, by Theorem~\ref{thm:triple_collision}, the dynamics of the ranked-market weights are given by \eqref{eqn:rank_based_dynamics}. Moreover, $X$ is ergodic with density $p$ and $X_{()}$ is ergodic with density $q$. 

This model recovers and extends models for the market weights previously studied in the SPT literature. The case $\beta > 0$ and $\sigma_i^2 = \sigma^2$ for every $i=1,\dots,d$ and some $\sigma^2 > 0$ recovers the \emph{generalized volatility-stabilized markets} introduced in \cite{pickova2014generalized}. The further specification $\beta = 1/2, \sigma^2=1$ and $\alpha_i = \alpha/2$ for some $\alpha > 0$ recovers the original \emph{volatility-stabilized market} introduced in \cite{fernholz2005relative}. When $\beta = 0$, we recover a subset of the models studied in \cite{ichiba2011hybrid}; namely those models with only name-based (as opposed to rank-based) variances. 

We can also handle the case when $\beta < 0$. In this situation $p$ may not be integrable. As such, for $\beta < 0$, we additionally impose the condition 
\[\frac{\alpha_i}{\sigma_i^2} \geq (1-2\beta)d-2; \quad i=1,\dots,d.\]
This condition guarantees that $p$ is integrable and that Assumption~\ref{ass:non_explosion} holds by virtue of Proposition~\ref{prop:suff_cond}. As such, the results of the previous sections apply yielding ergodicity and absence of triple collisions. This specification, however, does not seem to have a clear financial motivation. The case $\beta > 0$ yields the qualitative property ``larger stocks have larger variances", which is observed in equity markets, while the choice $\beta < 0$ would yield the reverse relationship. We expand on this qualitative property of generalized volatility stabilized markets in Remark~\ref{rem:SPT}.

\subsection{Atlas and Volatility-Stabilized Markets, Hybrid Drift} \label{eg:hybrid_drift}
	
	Let $\beta \geq 0$, $\sigma^2 > 0$ and $\gamma,g \in \R^d$ be given. For every $\tau \in \mathcal{T}$ and $x \in \bar \Delta^{d-1}_{+,\tau}$ define
	\begin{equation}\label{eqn:p_atlas}
		p_{\tau}(x)  = 
	\|x\|_{2\beta}^{2(1+(d-1)\beta)-d}\prod_{i=1}^d x_i^{\frac{2(\gamma_i + g_{\tau(i)})}{\sigma^2}+2\beta-1}.
	\end{equation}
We see that $p_\tau = p_{\tau'}$ for $x \in \partial \Delta^{d-1}_{+,\tau} \cap \partial \Delta^{d-1}_{+,\tau'}$ so we can define $p:\Delta^{d-1}_+ \to (0,\infty)$ by setting $p(x) = \sum_{\tau \in \mathcal{T}} p_\tau(x)1_{\Delta^{d-1}_{+,\tau}}(x)$ for $x \in \bigcup_{\tau} \Delta^{d-1}_{+,\tau}$ and continuously extending it to $\Delta^{d-1}_+$. Next, for $x \in \Delta^{d-1}_+$ define
\begin{align}
	c_{ij}(x) & = -\sigma^2x_ix_j\( x_i^{1-2\beta}+ x_j^{1-2\beta}- \sum_{k=1}^d  x_k^{2(1-\beta)}\), && i,j=1,\dots,d, \ i \ne j \label{eqn:c_atlas}
\end{align}
and $c_{ii}(x) = \sum_{j \ne i} c_{ij}(x)$ for $i =1,\dots,d$. Here for $r \ne 0$,  $\|x\|_{r} = (\sum_{i=1}^d x_i^{r})^{1/r}$ and, otherwise, we use the convention $\|x\|_0 = 1$ for every $x \in \Delta^{d-1}_+$. It is clear that Assumption~\ref{ass:inputs} is satisfied for this choice of $(c,p)$. Thus, the dynamics of the process $X$ corresponding to the (pre-)Dirichlet form defined by \eqref{eqn:Dirichlet_form} are given by \eqref{eqn:X_semimart_decomp}. For every $x \in \Delta^{d-1}_{+}$ and $i =1,\dots,d$ let $n_i(x) = k$, where $k$ is such that $r_k(x) = i$. With this notation in place we see that \eqref{eqn:X_semimart_decomp} becomes
 \begin{equation} \label{eqn:Atlas_market_dynam}
	\begin{split}
		\frac{dX_i}{X_i} = \((\gamma_i + g_{n_i(X)}+\frac{\sigma^2}{2})X_i^{-2\beta} - \sigma^2 X_i^{1-2\beta} - \sum_{j=1}^d (\gamma_j + g_{n_j(X)}+\frac{\sigma^2}{2})X_j^{1-2\beta} + \sigma^2\|X\|_{2(1-\beta)}^{2(1-\beta)}\)dt \\
		+   \sigma X_i^{-\beta}(1-X_i)dW_i + \sigma\sum_{j \ne i} X_j^{1-\beta}dW_j
	\end{split}
\end{equation}
for every $i=1,\dots,d$,
where we omitted the time argument for notational clarity. Additionally, by Theorem~\ref{thm:triple_collision} there are no triple collisions and the dynamics of the ranked market weights are given by \eqref{eqn:rank_based_dynamics}. It can be readily checked that under the stability condition
\begin{align}\label{eqn:stability_cond}
2\beta \ell +2\min\{1-\beta,0\} +	\frac{1}{\sigma^2}\sum_{k=1}^l (g_{d+1-k} + \gamma_{\tau(d+1-k)}) > 0; \qquad  l = 1,\dots,d-1, \  \tau \in \mathcal{T},
\end{align}
the condition \eqref{eqn:suff_cond} holds so that Assumption~\ref{ass:non_explosion} holds by virtue of Proposition~\ref{prop:suff_cond}. Moreover, under condition \eqref{eqn:stability_cond} $p$ is integrable so that $X$ and $X_{()}$ are ergodic by Proposition~\ref{prop:ergodic}. The invariant (unnormalized) density for the ranked market weights in this case is given by $q:\Delta^{d-1}_{+,\geq} \to (0,\infty)$ via

\begin{equation}\label{eqn:Atlas_density} q(y) = \|y\|_{2\beta}^{2(1+(d-1)\beta)-d} \sum_{\tau \in \mathcal{T}} \prod_{k=1}^d y_k^{\frac{2( \gamma_{\tau(k)}+ g_k)}{\sigma^2} + 2\beta -1}.
\end{equation}
The case $\gamma = 0$ yields a pure rank-based model, while the case $g = 0$ gives a fully name-based model and coincides with Example~\ref{eg:name_based_drift} where $\sigma_i^2 = \sigma^2$ for every $i$. 

The specification $\beta = 0$ recovers a sub-class of the so-called \emph{hybrid Atlas models} introduced in \cite{ichiba2011hybrid}; namely those models which have the same volatility parameter $\sigma^2$ for each rank. The further specification $g_k = -g$ for $k=1,\dots,d-1$ and $g_d = (d-1)g$ for some constant $g > 0$ recovers the original \emph{Atlas model} first introduced in \cite{banner2005atlas}. In the case $\beta = 0$ we can assume without loss of generality that $\sum_{i=1}^d \gamma_i + g_i = 0$, due to the fact that $\sum_{i=1}^d dX_i = 0$. Indeed, the input pair $(\gamma,g)$ and $(\gamma - \frac{1}{d}\sum_{i=1}^d \gamma_i \boldsymbol{1},g - \frac{1}{d}\sum_{i=1}^d g_i\boldsymbol{1})$ yield the same market weight dynamics \eqref{eqn:Atlas_market_dynam}. Under this assumption, condition \eqref{eqn:stability_cond} reduces to the stability condition \cite[Equation~(3.3)]{ichiba2011hybrid}, under which the authors of that paper obtain ergodicity of the ranked market weights. In this case the formula \eqref{eqn:Atlas_density} reduces to \cite[Equation~5.18]{ichiba2011hybrid} up to a normalizing constant.

 The case $\beta > 0$, to the best of our knowledge, has not previously been studied in the SPT literature. We view this specification as an extension of the generalized volatility-stabilized markets \cite{pickova2014generalized}, which allows for rank-based drift dependencies, in addition to name-based dependencies, for the market weights.
\begin{remark} \label{rem:SPT}
	Both the generalized volatility stabilized models of \cite{pickova2014generalized} and hybrid Atlas models of \cite{ichiba2011hybrid} first define the stock price $S$ and then obtain the market weights via the transformation $X_i=S_i/(S_1+\dots+S_d)$ as described above. The stock price processes introduced there, which generate the market weights of Examples~\ref{eg:name_based_drift} and \ref{eg:hybrid_drift}, are given by
	\begin{equation} \label{eqn:Atlas_stock}d\log S_i(t) = \frac{\varphi_i(X(t))}{X_i(t)^{2\beta}}dt + \frac{\sigma_i}{X_i(t)^\beta}dW_i(t), \quad i=1,\dots,d
	\end{equation}
	for specific choices of piecewise constant $\varphi_i$. 
	 Indeed, in Example~\ref{eg:name_based_drift}, $\varphi_i \equiv \alpha_i$, while in Example~\ref{eg:hybrid_drift} $\varphi_i(x) = \gamma_i + \sum_{k=1}^d g_k1_{\{r_k(x) = i\}}$ and $\sigma_i = \sigma$ for some $\sigma > 0$ and every $i=1,\dots,d$. When $\beta = 0$ and $\gamma = 0$ the latter reduces to the Brownian particle system of Example~\ref{eg:Brownian_particle}. The case $\beta > 0$ has the feature that the low ranked assets have higher volatilities, which captures a qualitative effect observed in equity markets. We chose to work directly with the market weights in Examples~\ref{eg:name_based_drift} and \ref{eg:hybrid_drift} rather than the stock capitalizations as we were able to more directly apply the ergodic results of Section~\ref{sec:ergodicty}. 
\end{remark}

\subsection{A Tractable Class of Models}
Suppose $p \in C^1_{\mathcal{T}}(\Delta^{d-1}_+;(0,\infty))$ is given with $\int_{\Delta^{d-1}_+} p = 1$. Take input functions $f_i, f_{ij}$ and $g$ which satisfy:

\begin{enumerate}[label = ({\roman*})] \itemsep0.1em
	\item For each $i=1,\dots,d$, $f_i \in C^1 ((0,1); (0,\infty))$ is bounded and satisfies $\lim_{y \downarrow 0} f_i(y) = 0$,
	\item $g \in C^1_{\mathcal{T}}(\Delta^{d-1}_{+};(0,\infty))$,
	\item For each $i,j=1,\dots,d$ with $i \ne j$ and $x \in \cup_{\tau} \Delta^{d-1}_{+,\tau}$ set $f_{ij}(x) = \sum_{\tau \in \mathcal{T}} f^{ij}_{\tau}(x_{-ij})1_{\Delta^{d-1}_{+,\tau}}(x)$ where $x_{-ij}$ is the $d-2$ dimensional vector obtained from $x$ by removing the $i^{\text{th}}$ and $j^{\text{th}}$ component. Here for each $\tau \in \mathcal{T}$,  $f^{ij}_{\tau}$ is bounded, satisfies $f^{ij}_{\tau} \in C^1(\tilde E^{d-2};(0,\infty))$ where $\tilde E^{d-2} = \{x \in (0,1)^{d-2}: \sum_{k=1}^{d-2}x_k < 1\}$ and is chosen such that $f_{ij}$ can be continuously extended to $\Delta^{d-1}_{+}$.	
\end{enumerate}
We then define for $x \in \Delta^{d-1}_+$, 
\begin{equation} \label{eqn:c_tractable}
c_{ij}(x) = -f_{ij}(x)f_i(x_i)f_j(x_j)g(x); \quad i \ne j
\end{equation} and $c_{ii}(x) = \sum_{j\ne i}^d c_{ij}(x)$ for $i=1,\dots,d$. The conditions (i)-(iii) on the input functions ensure that $c \in C^1_{\mathcal{T}}(\Delta^{d-1}_{+};\mathbb{S}^d_{++})$. Hence Assumption~\ref{ass:inputs} holds for this pair $(c,p)$.

This model is a natural extension of the tractable class of models introduced in \cite[Section~6]{itkin2020robust}.
The product structure of the covariation matrix $c$ in \eqref{eqn:c_tractable} is such that the \emph{supermartingale num\'eraire portfolio} in the equity market is \emph{functionally generated} in the sense of \cite[Chapter~3]{fernholz2002stochastic}. As such, this structure allows one to obtain explicit formulas for the growth-optimal portfolio in arbitrary dimension,  which is an attractive property for this class of models. Additionally, the specification $f_i(x) = x$, $g(x) =  1$ and $f_{ij}(x) = \alpha_{ij}$ for some constants $\alpha_{ij} \geq 0$ and every $x \in \Delta^{d-1}_+$ in \eqref{eqn:c_tractable} together with the choice of Dirichlet distribution for $p$,
\[p(x) \propto \prod_{i=1}^d x_i^{a_i-1}\] for some constants $a_i > 0$ results in $X$ being a \emph{polynomial process on the simplex} as studied in \cite{cuchiero2019polynomial}. Such processes have a very special structure which, among other things, allow for analytic computation of moments. We refer the reader to \cite{filipovic2016polynomial,cuchiero2019polynomial} for a detailed discussion of polynomial processes as well as their applications in SPT and other areas of mathematical finance. 

The results of this paper now show that the market weight process is strong Feller, no triple collisions occur in this model and, if  the non-explosion criterion Assumption~\ref{ass:non_explosion} is satisfied (which for example follows if \eqref{eqn:suff_cond} holds), then the market weight and ranked market weight processes are ergodic.

\section{Conclusion} \label{sec:conclusion}
Using the theory of Dirichlet forms we have constructed a well-behaved class of continuous semimartingales which evolve according to their relative ranks. Using the special structure of this construction we were able to obtain the strong Feller property, prove the nonexistence of triple collisions and various ergodic properties. As explained in the introduction and in Section~\ref{sec:SPT}, this class extends the analysis of rank-based diffusions beyond the case of competing Brownian particle systems previously studied in the literature. 

There are, however, clear limitations of this approach and we are unable to recover certain important specifications of competing Brownian particle systems. Indeed, consider again the diffusion mentioned in Example~\ref{eg:Brownian_particle},

\begin{equation} \label{eqn:Brownian_particle_rank_vol}
	dY_i(t) = \sum_{k=1}^d g_k1_{\{r_k(X(t)) = i\}}dt + \sum_{k=1}^d \sigma_k1_{\{r_k(X(t)) = i\}}dW_i(t)
\end{equation} where
$g \in \R^d$ and $\sigma \in (0,\infty)^d$. As mentioned in the discussion of Example~\ref{eg:Brownian_particle} the covariation matrix corresponding to this diffusion is $c(x) = \sum_\tau \sigma^2_\tau 1_{E_\tau}(x) I_{d \times d}$, which is discontinuous and so does not satisfy Assumption~\ref{ass:inputs}. Nevertheless if we define 
\[p(x) \propto \exp\(2\sum_{k=1}^d \frac{x_{(k)}g_k}{\sigma_k^2}\)\] then we have that $\frac{1}{2}\diver c_i +\frac{1}{2}c_i\nabla \log p = \sum_{k=1}^d g_k1_{\{r_k(\cdot) = i\}}$ so one might hope the Dirichlet form \eqref{eqn:Dirichlet_form} corresponding to this choice of $(c,p)$ and the domain $E = \R^d$ will still generate the process $Y$ despite the fact that $c$ is discontinuous. Alas, this is not the case. For test functions $u,v \in \mathcal{D}$ we have from \eqref{eqn:IBP_boundary} that 
\begin{align*}  
\mathcal{E}(u,v) & = \int_{\R^d} -vLup \  +  \sum_{(\tau,\tau') \in \partial\mathcal{T}^2} \int_{\partial E_\tau \cap \partial E_{\tau'}} v \nu_{\tau}^\top (c_\tau p_\tau - c_{\tau'}p_{\tau'})\nabla u \ d\mathcal{H}^{d-1}\\
& = \int_{\R^d} -vLup \  - \frac{d(d-1)}{2\sqrt{2}}\sum_{k=1}^{d-1}(\sigma_k^2 - \sigma_{k+1}^2) \int_{\{x_{(k)} = x_{(k+1)}\}} v(\partial_{r_k}u - \partial_{r_{k+1}}u)p \ d\mathcal{H}^{d-1}
\end{align*} 
where $L$, given by \eqref{eqn:generator}, is the generator of \eqref{eqn:Brownian_particle_rank_vol}. We see that the boundary terms in the integration by parts formula do not vanish in this case so it follows that the process corresponding to this Dirichlet form will have local time terms in its semimartingale decomposition, and so, cannot correspond to the process \eqref{eqn:Brownian_particle_rank_vol}.

Moreover, \cite[Theorem~2]{ichiba2011hybrid} establishes under the condition 
\begin{equation} \label{eqn:lin_vol}
\sigma_{2}^2 - \sigma_1^2 = \sigma_3^2 - \sigma_2^2 = \dots = \sigma_d^2 - \sigma_{d-1}^2
\end{equation}
an explicit formula for the marginal invariant density of the process $(Y_{(1)} - Y_{(2)},\dots,Y_{(d-1)} - Y_{(d)})$. As a consequence of that theorem it follows that under the condition \eqref{eqn:lin_vol} the process $Y$, given by \eqref{eqn:Brownian_particle_rank_vol}, has invariant density  
\[\rho(x) \propto \exp\(4\sum_{k=1}^{d} x_{(k)}(\lambda_k - \lambda_{k-1})\)\]
where \[\lambda_k = \frac{\sum_{l=1}^{k}g_l}{\sigma_k^2 + \sigma_{k+1}^2 }; \quad k=1,\dots,d-1\] and $\lambda_0 = \lambda_{d} = 0$. Aside from the case where all of the $\sigma_k$'s are the same, $L$ is \emph{not} a symmetric operator on $L^2(\R^d;\nu)$, where $d\nu(x) = \rho(x)dx$. 
 Hence, it follows that one cannot construct the process $X$ in \eqref{eqn:Brownian_particle_rank_vol} using a (symmetric) Dirichlet form.

Additionally, the question of triple collisions is more delicate in the case of a rank-based, discontinuous volatility structure. We have shown in Theorem~\ref{thm:triple_collision} that triple collisions never occur for the class of processes constructed in this paper. The authors of \cite{MR3055258,MR3325099}, however, show that for the process \eqref{eqn:Brownian_particle_rank_vol} triple collisions do not occur with probability one if and only the concavity condition
\[\frac{1}{2}(\sigma_{k+1}^2 + \sigma_{k-1}^2) \leq \sigma_k^2\]
holds for every $k = 2,\dots,d-1$; in particular triple collisions do occur with positive probability for some specifications. As such, it is clear that the approach undertaken in this paper cannot handle such cases. It remains an open problem to find general methods that establish the existence/non-existence of triple collisions for diffusions with discontinuous volatility coefficients beyond the case of competing Brownian particle systems.

\bibliographystyle{plain}
\bibliography{references}	
	
\end{document}